\documentclass[11pt,reqno]{amsart}
\usepackage{amsmath, amsthm, amssymb}
\usepackage{mathrsfs}
\usepackage{array}
\usepackage{caption}
\evensidemargin \oddsidemargin
\def\bfy{{\mathbf y}}
\def\bfx{{\mathbf x}}

\newcommand{\mmod}[1]{\,\,(\text{\rm mod}\,\, #1)}

\def\bfalpha{{\boldsymbol \alpha}}
\def\bfbeta{{\boldsymbol \beta}}

\newtheorem{thm}{Theorem}
\newtheorem{cor}{Corollary}

\newtheorem{lem}{Lemma}

\newtheorem{prop}{Proposition}

\numberwithin{equation}{section} \numberwithin{thm}{section}
\numberwithin{lem}{section} \numberwithin{problem}{section}
\numberwithin{cor}{section}

\def\grm{{\mathfrak m}}\def\grM{{\mathfrak M}}\def\grN{{\mathfrak N}}\def\grn{{\mathfrak n}}\def\grp{{\mathfrak p}}\def\grP{{\mathfrak P}}\def\grw{{\mathfrak w}}\def\grW{{\mathfrak W}}

\begin{document}
\title[Sums of three cubes]{On Waring's problem in sums of three cubes}
\author[Javier Pliego]{Javier Pliego}
\address{Purdue Mathematical Science Building, 150 N University St, West Lafayette, IN 47907, United States of America}

\email{jp17412@bristol.ac.uk}
\subjclass[2010]{11P05, 11P55}
\keywords{Waring's problem, Hardy-Littlewood method.}

\begin{abstract} We investigate the asymptotic formula for the number of representations of a large positive integer as a sum of $k$-th powers of integers represented as the sums of three positive cubes, counted with multiplicities. We also obtain a lower bound for the number of representations when the sums of three cubes are counted without multiplicities.
\end{abstract}
\maketitle

\section{Introduction} 
It is widely believed, but still unknown, that the set of integers $\mathscr{C}$ represented as a sum of three positive integral cubes has positive density. Hardy and Littlewood \cite{Har} first announced what is known as the Hypothesis-$K$, which asserts that for each $\varepsilon>0$, the number of representations $r_{k}(n)$ of $n$ as a sum of $k$ positive integral $k$-th powers is $O(n^{\varepsilon}).$ Although this conjecture is known to be false when $k=3$ (see Mahler \cite{Mah}), the weaker claim that \begin{equation}\label{ec1.1}\sum_{n\leq X}r_{k}(n)^{2}\ll X^{1+\varepsilon},\end{equation} known as Hypothesis $K^{*}$ (see \cite{Hol2}), would allow one to show, through a standard Cauchy-Schwarz argument, that $\mathcal{N}(X)=\lvert \mathscr{C}\cap [1,X]\rvert\gg X^{1-\varepsilon}.$ In fact, under some unproved assumptions on the zeros of some Hasse-Weil $L$-functions, Hooley (\cite{Hol1}, \cite{Hol2}) and Heath-Brown \cite{Hea} showed using different procedures that (\ref{ec1.1}) holds for $k=3$. Nevertheless, some unconditional progress has been made on strengthening  lower bounds for $\mathcal{N}(X).$ By using methods of diminishing ranges, Davenport \cite{Dav2} obtained the bound $\mathcal{N}(X)\gg X^{47/54-\varepsilon}.$
Later on, Vaughan improved it to $\mathcal{N}(X)\gg X^{11/12-\varepsilon}$ by introducing smooth numbers in his ``new iterative method'' \cite{Vau3}, and Wooley, extending the method to obtain non-trivial bounds for fractional moments of smooth Weyl sums, improved the estimate in a series of papers (\cite{Woo1}, \cite{Woo2}, \cite{Woo3}), the best current one being $\mathcal{N}(X)\gg X^{\beta}$, where $\beta=0.91709477.$

A vast number of results can be found in the literature on problems involving equations over special subsets of the integers. The Green-Tao Theorem \cite{G-T}, which proves the existence of arbitrarily long arithmetic progressions over the primes is an example of such problems when the special set is the set of prime numbers. Other instances where the set $\mathscr{C}$ is involved include some correlation estimates for sums of three cubes by Br\"udern and Wooley \cite{B-W}, and lower bounds of the shape $N_{3}(\mathscr{C},X)\gg X^{5/2-\varepsilon},$ by Balog and Br\"udern \cite{B-B}. The parameter $N_{3}(\mathscr{C},X)$ here denotes the number of triples with entries in $\mathscr{C}\cap [1,X]$ whose entries averages lie on $\mathscr{C}$ as well. 

In this paper we investigate the asymptotic formula for Waring's problem when the set of $k$-th powers of integers is replaced by the set of $k$-th powers of elements of $\mathscr{C}$, but before stating the main result that we obtain here it is convenient to introduce some notation. Let $k\geq 2$ and $n\in\mathbb{N}$. Take $P=n^{1/3k}.$ For every vector $\mathbf{v}\in\mathbb{R}^{n}$ and parameters $a,b\in\mathbb{R}$ we will write $a\leq \mathbf{v}\leq b$ to denote that $a\leq v_{i}\leq b$ for $1\leq i\leq n$. We take the function $T(\mathbf{x})=x_{1}^{3}+x_{2}^{3}+x_{3}^{3},$ and consider the weights $$r_{3}(x)={\rm card}\Big\{\bfx\in\mathbb{N}^{3}:\ x=T(\bfx),\ \ \mathbf{x}\leq P\Big\}$$ and the set $$\mathcal{X}_{n}=\Big\{(x_{1},\ldots,x_{s})\in\mathscr{C}^{s},\ \ \ \ n=\sum_{i=1}^{s}x_{i}^{k}\Big\}.$$ Define the functions \begin{equation}\label{Rs}R(n)=\sum_{\mathbf{X}\in\mathcal{X}_{n}}r_{3}(x_{1})\cdots r_{3}(x_{s}),\ \ \ \ \ \ \ \ \ \ \ \ \  r(n)=\sum_{\mathbf{X}\in\mathcal{X}_{n}}1,\end{equation} which count the number of representations of $n$ as a sum of $k$-th powers of integers represented as sums of three positive cubes, counted with and without multiplicities respectively. Take the singular series associated to the problem, defined as
\begin{equation}\label{1A}\mathfrak{S}(n)=\sum_{q=1}^{\infty}\sum_{\substack{a=1\\ (a,q)=1}}^{q}\Big(q^{-3}\sum_{1\leq \mathbf{r}\leq q}e\big(aT(\mathbf{r})^{k}/q\big)\Big)^{s}e\big(-an/q\big).\end{equation}
The main result of this paper establishes an asymptotic formula for $R(n)$. For such purpose, it is convenient to introduce the parameter $H(k)=9k^{2}-k+2.$
\begin{thm}\label{thm9.1}
Let $s\geq H(k).$ Then, there exists a constant $\delta>0$ such that
\begin{equation*}R(n)=\Gamma\big(4/3\big)^{3s}\Gamma\big(1+1/k\big)^{s}\Gamma\big(s/k\big)^{-1}\mathfrak{S}(n)n^{s/k-1}+O(n^{s/k-1-\delta}),\end{equation*}where the singular series satisfies $\mathfrak{S}(n)\gg 1.$ 
\end{thm}

Our proof of Theorem \ref{thm9.1} is based on the application of the Hardy-Littlewood method. In order to discuss the constraint of the previous result on the number of variables, we define first $\tilde{G}(k)$ as the minimum number such that for $s\geq \tilde{G}(k)$, the anticipated asymptotic formula in the classical Waring's problem holds. We remind the reader that as a consequence of Vinogradov's mean value theorem, Bourgain \cite{Bou} showed that $\tilde{G}(k)\leq k^{2}-k+O(\sqrt{k})$. The lack of understanding of the cardinality of the set $\mathscr{C}$ mentioned at the beginning of the paper both weakens the minor arc bounds and prevents us from having a better understanding of its distribution over arithmetic progressions, which often comes into play on the major arc analysis. The methods used in this memoir then are based on arguments in which in most of the sums of three cubes employed in the representation, all but one of the cubes is fixed in the associated analysis. Consequently, the constraint for the number of variables that we obtain here is asymptotic to the bound for $\tilde{G}(3k)$ mentioned above.

 The problem becomes more challenging when we remove the counting of the multiplicities, and even if getting an asymptotic formula seems out of reach, Theorem \ref{thm9.1} can be used to obtain a non-trivial lower bound. However, the whole strategy relies on an estimate for the $L^{2}$-norm of the sequence $\big(r_{3}(x)\big)_{x\leq X}$ of the shape
\begin{equation}\label{rap}\displaystyle\sum_{x\leq X}r_{3}(x)^{2}\ll X^{7/6+\varepsilon}\end{equation} that follows after an application of Hua's Lemma \cite[Lemma 2.5]{Vau}. Instead of taking that approach, we restrict the triples to lie on $\mathcal{C}(P)=\big\{\mathbf{x}\in [1,P]^{3}:\ x_{1},x_{2}\in\mathcal{A}(P,P^{\eta})\big\}$, where $\eta>0$ is a small enough fixed parameter and 
$$\mathcal{A}(X,R)=\{n\in [1,X]\cap \mathbb{N}: p\mid n\text{ and $p$ prime}\Rightarrow p\leq R\},$$ and make use of the stronger estimate
\begin{equation}\label{ec12}\sum_{x\leq X}s_{3}(x)^{2}\ll X^{1+\nu}\end{equation} due to Wooley \cite[Theorem 1.2]{Woo3}, where $s_{3}(x)={\rm card}\big\{\bfx\in\mathcal{C}(P):\ x=T(\bfx)\big\}$ and $\nu=0.08290523$. It transpires that one should then have some control of the order of magnitude of the analogous function of $R(n)$ when we impose that restriction on the triples. For such matters, we define for each $n\in\mathbb{N}$ the aforementioned counting function
\begin{equation}\label{Rseta} R_{\eta}(n)=\sum_{\mathbf{X}\in\mathcal{X}_{n}}s_{3}(x_{1})\cdots s_{3}(x_{s}).\end{equation}We also introduce Dickman's function, defined for real $x$ by
$$\rho(x)=0\text{ when } x<0,$$
$$\rho(x)=1 \text{ when } 0\leq x\leq 1,$$
$$\rho \text{ continuous for } x>0,$$
$$\rho \text{ differentiable for } x>1$$
$$x\rho'(x)=-\rho(x-1) \text{ when } x>1.$$
\begin{thm}\label{thm9.3}
Let $s$ be any positive integer with $s\geq H(k).$ Then, there exists $\delta>0$ such that
$$R_{\eta}(n)=\Gamma\big(4/3\big)^{3s}\Gamma\big(1+1/k\big)^{s}\Gamma\big(s/k\big)^{-1}\rho\big(1/\eta\big)^{2s}\mathfrak{S}(n)n^{s/k-1}+O(n^{s/k-1}(\log n)^{-\delta}),$$
where the singular series satisfies $\mathfrak{S}(n)\gg 1.$ 
\end{thm}
An application of this theorem then, together with equation (\ref{ec12}) and some other arguments yield the following result, which improves substantially the bound that one could obtain if no restriction on the triples was made.
\begin{thm}\label{thm9.2}
Let $s$ be any positive integer with $s\geq H(k)+1.$ One has the lower bound
$$r(n)\gg n^{(1-\nu)s/k-1},$$ where $\nu$ was defined right after (\ref{ec12}).
\end{thm}
It is worth noting that the preceding lower bound may be the best possible
estimate attainable with the current knowledge available. The final question that will be addressed here is the constraint on the number of variables that guarantees the existence of solutions. For such purpose, we define $G_{3}(k)$ as the minimum integer such that for all $s\geq G_{3}(k)$ then $r(n)\geq 1$ holds for sufficiently large integers. We apply a previous result of Wooley \cite{Woo7} to obtain the following bound.
\begin{thm}\label{thm1.4}
Let $k\in\mathbb{N}$. Then,
$$G_{3}(k)\leq 3k\big(\log k+\log\log k+O(1)\big).$$
\end{thm}
Our proofs for the main theorems of the paper are based on the application of the Hardy-Littlewood method. In Section \ref{sec2}, we apply a mean value estimate related to that of Vinogradov to bound the minor arc contribution. Section \ref{sec3} deals with estimates of complete exponential sums and other related sums. In Sections \ref{sec4} we discuss the local solubility of the problem and some properties of the singular series and include a brief proof of Theorem \ref{thm1.4}. Using the Riemann-Stieltjes integral we give an approximation of $f(\alpha)$ over the major arcs in Section \ref{sec6}. In Section \ref{sec7} we study the singular integral, we obtain an asymptotic formula for the major arcs and we include a proof of Theorem \ref{thm9.1}. Section \ref{Sec99} is devoted to the study of the asymptotic formula when we introduce smooth numbers, and Theorem \ref{thm9.2} is then proven in Section \ref{Sec999} via an application of Theorem \ref{thm9.3}. We have also included a small appendix in which we improve the constraint on the number of variables needed in Theorem \ref{thm9.1} for small exponents by using restriction estimates.

\emph{Notation}.  Whenever $\varepsilon$ appears in any bound, it will mean that the bound holds for every $\varepsilon>0$, though the implicit constant may depend on $\varepsilon$. We adopt the convention that when we write $\delta$ in the computations we mean that there exists a positive constant such that the bound holds. Unless specified, any lower case letter $\mathbf{x}$ written in bold will denote a triple of integers $(x_{1},x_{2},x_{3})$. For any scalar $\lambda$ and any vector $\mathbf{x}$ we write $\lambda \mathbf{x}$ for the vector $(\lambda x_{1},\lambda x_{2}, \lambda x_{3})$. When $R,V\in\mathbb{Z}^{d}$ then $R\equiv V\pmod q$ will mean that $R_{i}\equiv V_{i}\pmod{q}$ for all $1\leq i\leq d$. We use $\ll$ and $\gg$ to denote Vinogradov's notation, and write $A\asymp B$ whenever $A\ll B\ll A$. As usual in analytic number theory, for each $x\in\mathbb{R}$ then $e(x)$ will mean $\exp(2\pi i x),$ and for each prime $p$, the number $e(x/p)$ will be written as $e_{p}(x).$ We write $p^{r}|| n$ to denote that $p^{r}| n$ but $p^{r+1}\nmid n.$

\section{Minor arc estimate.}\label{sec2}
We obtain estimates for certain moments of an exponential sum on the minor arcs which we now define. Fix $s,k\geq 2$ and consider
\begin{equation*}f(\alpha)=\sum_{\bfx\leq P}f_{\bfx}(\alpha),\ \ \ \ \ \ \  \ \text{where}\ \  f_{\bfx}(\alpha)=\sum_{1\leq x\leq P}e\big(\alpha T(\bfx,x)^{k}\big)\end{equation*} and $\mathbf{x}\in\mathbb{N}^{2}$. Recalling (\ref{Rs}), note that by orthogonality it follows that
$$R(n)=\int_{0}^{1}f(\alpha)^{s}e(-\alpha n)d\alpha.$$ The purpose of this section is to bound the minor arc contribution of this integral. In order to make further progress we make use of a Hardy-Littlewood dissection in our analysis. When $a\in\mathbb{Z}$ and $q\in\mathbb{N}$ satisfy $0\leq a\leq q\leq P^{\xi}$ and $(a,q)=1$ with $\xi<\frac{s}{s+2}$, consider 
\begin{equation}\label{5us}\grM(a,q)=\Big\{ \alpha\in [0,1):\Big\lvert \alpha-a/q\Big\rvert \leq \frac{P^{\xi}}{qn} \Big\}.\end{equation} Then the major arcs $\grM$ will be the union of these arcs and $\grm=[0,1)\setminus \grM $ will be the minor arcs. 
\begin{prop}\label{prop222} When $s$ is any positive integer with $s\geq H(k)$ one has
$$\int_{\grm}\lvert f(\alpha)\rvert^{s} {\rm     d}\alpha\ll P^{3s-3k-\delta}.$$
Moreover, if $s\geq 3k(3k+1)$ then it follows that
$$\int_{\grm}\lvert f(\alpha)\rvert^{s} {\rm     d}\alpha\ll P^{3s-3k-\xi+\varepsilon}.$$
\end{prop}
\begin{proof}
We bound the previous integrals in terms of a mean value of that of Vinogradov and apply estimates derived from Wooley \cite[Theorems 14.4, 14.5]{Woo4}. For such purpose, it is convenient to take the set $\frak{B}=\grm\times [0,1)^{k-1}$ and consider the exponential sums \begin{equation}\label{GF}G_{\mathbf{x}}(\bfalpha)=\displaystyle\sum_{ x\leq P}e\big(\alpha_{k} T(x,\mathbf{x})^{k}+\displaystyle\sum_{j=1}^{k-1}\alpha_{j}x^{3j}\big)\ \ \ \ \text{and} \ \ \ \ F(\bfalpha)=\displaystyle\sum_{x\leq P} e\Big(\sum_{j=1}^{k}\alpha_{j}x^{3j}\Big).\end{equation} We write $H(k)=2t$ for some positive integer $t$. Using H\"older's inequality and orthogonality we find that
\begin{align}\label{puu}&
\int_{\grm}\lvert f(\alpha)\rvert^{2t}d\alpha\ll  P^{4t-2}\int_{\grm}\sum_{\mathbf{x}\leq P}\lvert f_{\mathbf{x}}(\alpha)\lvert^{2t}d\alpha\nonumber
\\
&=P^{4t-2}\sum_{\substack{\mathbf{x}\leq P}}\sum_{n_{j}}\int_{\frak{B}}\lvert G_{\mathbf{x}}(\bfalpha)\lvert^{2t}e\big(-\sum_{j=1}^{k-1}\alpha_{j} n_{j}\big)d\bfalpha\ll P^{4t+3k(k-1)/2}\int_{\frak{B}}\lvert F(\bfalpha) \rvert^{2t}d\bfalpha,
\end{align}
where $(n_{j})_{j}$ runs over the tuples with $1\leq \lvert n_{j}\rvert\leq tP^{3j}.$ Observe that by Weyl's inequality \cite[Lemma 2.4]{Vau} one has that
\begin{equation*}\sup_{\bfalpha\in\frak{B}}\lvert F(\bfalpha)\rvert\ll P^{1-\delta},\end{equation*}
whence this pointwise bound and Theorem 14.5 of \cite{Woo4} with the choice $r=3k-2$ deliver the estimate
\begin{equation}\label{BF}\int_{\frak{B}}\lvert F(\bfalpha)\rvert^{2t}d\bfalpha\ll P^{2t-3k(k+1)/2-\delta}.\end{equation} The above equation and (\ref{puu}) then yield the first part of the proposition. For the second part we use a small modification of Wooley \cite[Theorem 14.4]{Woo4}. On that paper, the author, in a more general setting, takes the choice $\xi=1$ and obtains a saving of $X$ over the expected main term. It transpires that the same exact method can be applied to save $X^{\xi}$ for $\xi<1$. Thus, we have that for $s\geq 3k(3k+1)$ then
$$\int_{\frak{B}}\lvert F(\bfalpha)\rvert^{s}d\bfalpha\ll P^{s-3k(k+1)/2-\xi+\varepsilon}.$$ Replacing $2t$ by $s$ in (\ref{puu}) and using the previous equation we get the desired result.
\end{proof}
\section{Complete exponential sums}\label{sec3}
In this section we study the complete exponential sum associated to the problem and deduce some bounds involving this sum. For such purpose, it is convenient to define for $a\in\mathbb{Z}$ and $q\in\mathbb{N}$ with $(a,q)=1$ the expressions
\begin{equation*}S(q,a)=\sum_{1\leq\mathbf{r}\leq q}e_{q}\big(aT(\mathbf{r})^{k}\big)\ \ \ \ \ \ \ \text{and}\ \ \ \ \ \ \ S_{k}(q,a)=\sum_{r=1}^{q}e_{q}(ar^{k}).\end{equation*}Note that by orthogonality then one can rewrite $S(q,a)$ as
\begin{equation}\label{Sqa}S(q,a)=q^{-1}\sum_{u=1}^{q}S_{3}(q,u)^{3}S_{k}(q,a,-u),\ \ \ \ \ \ \text{where}\  \ S_{k}(q,a,b)=\sum_{r=1}^{q}e_{q}(ar^{k}+br).\end{equation}In what follows we provide bounds for $S(q,a)$ using estimates for $S_{3}(q,a)$ and $S_{k}(q,a,b)$. Observe that by the quasi-multiplicative structure of it then it suffices to investigate the instances when $q=p^{l}$ is a prime power. 

\begin{lem}\label{Sa}
Let $l\geq 2$, let $p$ be a prime number and $a\in\mathbb{Z}$ with $(a,p)=1.$ Then,
\begin{equation*}S(p^{l},a)\ll \min(p^{3l-1},lp^{3l-l/k+\varepsilon}).
 \end{equation*}
\end{lem}
\begin{proof}
Note that Vaughan \cite[Theorem 7.1]{Vau} yields the bound $S(p^{l},a,-u)\ll p^{l(1-1/k)+\varepsilon}$. Therefore, an application of this estimate and Theorem 4.2 of \cite{Vau} to equation (\ref{Sqa}) gives
\begin{equation*}S(p^{l},a)\ll p^{2l-l/k+\varepsilon}\sum_{u=1}^{p^{l}}(u,p^{l})\ll lp^{3l-l/k+\varepsilon}.\end{equation*}
Observe that we can also deduce the bound $S(p^{l},a,-u)\ll p^{l-1}$ from the proof\footnote{See in particular the argument following Vaughan \cite[(7.16)]{Vau}} of Vaughan \cite[Theorem 7.1]{Vau}, so the application of this estimate instead and the same procedure delivers $S(p^{l},a)\ll p^{3l-1}$.
\end{proof}
When $p$ is prime we can provide a more precise description of $S(p,a)$ by involving the sum $S_{k}(p,a)$ in its expression. Despite not using this refinement in the memoir, we have included such analysis for future work.
\begin{lem}\label{lem3.4}
Let $p$ be a prime number and $a\in\mathbb{Z}$ with $(a,p)=1$. Then,
$$S(p,a)=p^{2}S_{k}(p,a)+O(p^{2}).$$ In particular, one has the bound $S(p,a)\ll p^{5/2}.$
\end{lem}
\begin{proof}
By equation (\ref{Sqa}) it follows that $S(p,a)=p^{2}S_{k}(p,a)+E$, where $$E=p^{-1}\sum_{1\leq u\leq p-1}S_{3}(p,u)^{3}S_{k}(p,a,-u).$$
Using Vaughan \cite[Lemma 4.3]{Vau} to bound $S_{3}(p,u)$ and the work of Weil\footnote{See Schmidt \cite[Corollary 2F]{Sch} for an elementary proof of this bound.} \cite{Wei2} to bound $S_{k}(p,a,-u)$ we obtain the estimate $E\ll p^{2}$. Consequently, another application of the aforementioned lemma of Vaughan \cite{Vau} to $S_{k}(p,a)$ delivers $S(p,a)\ll p^{5/2}.$
\end{proof}
The reader may notice that this result is best possible since whenever $(k,p-1)>1$ then there is a positive proportion of positive integers $a\leq p$ for which $S_{k}(p,a)\gg p^{1/2}$, whence the above result delivers an asymptotic formula in those situations. It seems unclear whether the error term in the formula could be improved. Such improvement though would not have any impact in our work. For future purposes, it is convenient to define, for each $q\in\mathbb{N}$, the exponential sums
\begin{equation}\label{kkk}
S_{n}(q)=\sum_{\substack{a=1\\ (a,q)=1}}^{q}\big(q^{-3}S(q,a)\big)^{s}e_{q}(-na),
\ \ \ \ S_{s}^{*}(q)=\displaystyle\sum_{\substack{a=1\\ (a,q)=1}}^{q}\big\lvert q^{-3}S(q,a)\big\rvert^{s}\end{equation} 
and to analyse their behaviour when summing over $q$.
\begin{lem}\label{cor4}
Let $s\geq \max(4,k+1)$. One has
\begin{equation}\label{ssn}\sum_{q\leq Q}S_{s}^{*}(q)\ll Q^{\varepsilon}\ \ \ \ \ \ \ \text{and}\ \ \ \ \ \ \ \sum_{q\leq Q}\lvert S_{n}(q)\rvert\ll Q^{\varepsilon},\end{equation} and for $s\geq \max(5, k+2)$ it follows that
\begin{equation}\label{Sas}\sum_{q\leq Q}q^{1/k}\lvert S_{n}(q)\rvert\ll Q^{\varepsilon}\ \ \ \ \ \ \ \text{and}\ \ \ \ \ \ \ \sum_{q>Q}\lvert S_{n}(q)\rvert\ll Q^{\varepsilon-1/k}.\end{equation}
\end{lem}
\begin{proof}
To show (\ref{ssn}) it suffices to prove the bound for $S_{s}^{*}(q)$ since $\lvert S_{n}(q)\rvert\leq S_{s}^{*}(q)$. Applying Lemmata \ref{Sa} and \ref{lem3.4} we deduce trivially that each of $S_{n}(p)$ and $S_{s}^{*}(p)$ is $O(p^{1-s/2}),$ and each of $S_{n}(p^{l})$ and $S_{s}^{*}(p^{l})$ is $O\big(\min(p^{l-s},l^{s}p^{l-ls/k+\varepsilon})\big)$ when $l\geq 2$. Consequently, whenever $s\geq \max(4,k+1)$ then using the fact that $S_{s}^{*}(q)$ is multiplicative we find that
\begin{align*}\sum_{q\leq Q}S_{s}^{*}(q)\ll \prod_{p\leq Q}\big(1+\sum_{l=1}^{\infty}S_{s}^{*}(p^{l})\big)\ll \prod_{p\leq Q}(1+C/p)\ll Q^{\varepsilon},
\end{align*}
where $C>0$ is some suitable constant. The first assertion of (\ref{Sas}) follows by the same argument, and the second follows observing that then $$\sum_{Q\leq q\leq 2Q}\lvert S_{n}(q)\rvert\ll Q^{\varepsilon-1/k},$$ whence summing over dyadic intervals we obtained the desired result.
\end{proof}
\section{Singular series}\label{sec4}
We give sufficient conditions in terms of the number of variables to ensure the local solubility of the problem and combine such work with the bounds obtained in the previous section to introduce and analyse the singular series associated to the problem. We also include a brief proof of Theorem \ref{thm1.4}. For such purposes, a little preparation is required. Let $p$ a prime number and take $\tau\geq 0$ such that $p^{\tau}|| 3k$. Let $\gamma=2\tau+1,$ consider the set
\begin{equation}\label{after}\mathcal{M}_{n}(p^{h})=\Big\{\mathbf{Y}\in [1,p^{h}]^{3s}:\ \sum_{i=1}^{s}T(\bfy_{i})^{k}\equiv n\pmod{p^{h}}\Big\},\end{equation} where $\mathbf{Y}=(\mathbf{y}_{1},\ldots,\mathbf{y}_{s})$ with $\mathbf{y}_{i}\in\mathbb{N}^{3},$ and the subset $$\mathcal{M}_{n}^{*}(p^{h})=\Big\{\mathbf{Y}\in \mathcal{M}_{n}(p^{h}):\  p\nmid y_{1,1},\ p\nmid T(\mathbf{y}_{1})\Big\},$$ 
where $\mathbf{y}_{1}=(y_{1,1},y_{1,2},y_{1,3}).$ Define as well the quantities $M_{n}(p^{h})=\lvert\mathcal{M}_{n}(p^{h})\rvert$ and ${M}_{n}^{*}(p^{h})=\lvert\mathcal{M}_{n}^{*}(p^{h})\rvert.$
Here the reader may want to observe that the divisibility restrictions on the above definition are imposed for a latter application of Hensel's Lemma. Before showing that under some constraint in the number of variables then ${M}_{n}^{*}(p^{\gamma})>0$, we first provide an accurate description of the set
\begin{equation*}\mathcal{M}_{3,3}(p^{h})=\Big\{T(\bfx):\ \bfx\in\Big(\mathbb{Z}/p^{h}\mathbb{Z}\Big)^{3},\ (x_{1},p)=1\Big\}\end{equation*} that will be used throughout the whole argument. 
\begin{lem}\label{lem11.3}
Let $h\in\mathbb{N}.$ Then, whenever $p\neq 3$ one finds that \begin{equation}\label{M33}\mathcal{M}_{3,3}(p^{h})=\mathbb{Z}/p^{h}\mathbb{Z}.\end{equation} For the case $p=3$ one has $\mathcal{M}_{3,3}(3)=\mathbb{Z}/3\mathbb{Z}$ and when $h\geq 2$ then $$\mathcal{M}_{3,3}(3^{h})=\Big\{x\in \mathbb{Z}/3^{h}\mathbb{Z}:\ \ x\not\equiv 4\mmod{9},\ x\not\equiv 5\mmod{9}\Big\}.$$
\end{lem}
\begin{proof}
When $p\neq 3$, we can assume that $h=1$, since an application of Hensel's Lemma would then yield the case $h\geq 2$. For a better description of the argument, it is convenient to define the counting functions 
\begin{equation*}N_{n}(p)={\rm card}\Big\{\mathbf{x}\in\big(\mathbb{Z}/p\mathbb{Z}\big)^{3}:\ T(\mathbf{x})\equiv n\pmod{p}\Big\},\end{equation*}
\begin{equation*}N_{n,4}(p)={\rm card}\Big\{\mathbf{y}\in\big(\mathbb{Z}/p\mathbb{Z}\big)^{4}:\ y_{1}^{3}+y_{2}^{3}+y_{3}^{3}-ny_{4}^{3}\equiv 0\pmod{p}\Big\}.\end{equation*}Observe that by making a distinction for the tuples counted in $N_{n,4}(p)$ regarding the divisibility of $y_{4}$ by $p$ one has that $(p-1)N_{n}(p)=N_{n,4}(p)-N_{0}(p)$. Note that when $(n,p)=1$ then the work of Weil \cite{Wei} on equations over finite fields leads to
$$\lvert N_{n,4}(p)-p^{3}\rvert\leq 6(p-1)p\ \ \ \ \ \ \ \text{and}\ \ \ \ \ \ \ \lvert N_{0}(p)-p^{2}\rvert\leq 2(p-1)\sqrt{p}.$$ Consequently, one finds that $N_{n}(p)=p^{2}+E_{p}$ with $\lvert E_{p}\rvert\leq 6p+2\sqrt{p},$
and hence $N_{n}(p)\geq 1$ for $p\geq 7$. Observe as well that $N_{n}(2)\geq 1$ and $N_{n}(5)\geq 1$ follow trivially. This implies that there is at least one solution to the equation \begin{equation}\label{ec11.8}x^{3}+y^{3}+z^{3}\equiv n\pmod{p},\ \ \ \ \ \ (x,p)=1,\end{equation} and when $n=0$ then $(1,-1,0)$ is also a solution for (\ref{ec11.8}), whence the preceding discussion yields (\ref{M33}).

When $p=3$ then the case $h=1$ is trivial. Note that since cubes can only be $\pm 1\pmod{9}$, the only residues which cannot be written as sums of three cubes are $4$ and $5$. For $h=3$, a slightly tedious computation reveals that the only residues not represented as sums of three cubes are the ones congruent to $4$ or $5\pmod{9}$. Therefore, a routine application of Hensel's Lemma delivers the proof for $h\geq 4.$

\end{proof}

The previous lemma asserts that the local solubility of the problem studied here only differs from the local solubility of the original Waring's problem at the prime $3.$ This conclusion is gathered in the following statement.
\begin{lem}\label{lem9.1}
Suppose that $s\geq \frac{p}{p-1}\big(k,p^{\tau}(p-1)\big)$ when $p\neq 2,3$ or $p=2$ and $\tau=0$, that $s\geq \frac{9}{4}\big(k,\phi(3^{\gamma})\big)$ when $p=3$, that $s\geq 2^{\tau+2}$ when $p=2$ and $\tau>0$ with $k>2$, and that $s\geq 5$ when $p=k=2.$ Then one has $M_{n}^{*}(p^{\gamma})>0.$
\end{lem}
\begin{proof}
If $p\neq 3$ then Lemma \ref{lem11.3} implies that the local solubility for each of these primes is equivalent to that of the original Waring's problem, hence Vaughan \cite[Lemma 2.15]{Vau} yields $M_{n}^{*}(p^{\gamma})>0.$ Here the reader might want to observe that the definition for $\gamma$ taken here is different from the one in Vaughan \cite[(2.25)]{Vau}, so one may have to apply Lemma 2.13 of \cite{Vau} as well. For the case $p=3$, Lemma \ref{lem11.3} delivers $$\Big\lvert\mathcal{M}_{3,3}(3^{\gamma})\cap U(\mathbb{Z}/3^{\gamma}\mathbb{Z})\Big\rvert=4\cdot 3^{\gamma-2},$$ where $U(\mathbb{Z}/3^{\gamma}\mathbb{Z})$ denotes the group of units of $\mathbb{Z}/3^{\gamma}\mathbb{Z}$. Therefore, using Vaughan \cite[Lemma 2.14]{Vau} we get that
$M_{n}^{*}(3^{\gamma})>0$ whenever $s\geq \frac{9}{4}\big(k,\phi(3^{\gamma})\big).$ 
\end{proof}

Observe that by the combination of Lemma \ref{lem11.3} and Vaughan \cite[Lemma 2.14]{Vau} then for $u\geq 9k/4$ we find that the form $T(\mathbf{x}_{1})^{k}+\dots+T(\mathbf{x}_{u})^{k}$ covers all the residue classes modulo $3^{k}$. Take now $s$ such that every sufficiently large number can be written as a sum of $s$ integral $3k$-th powers. Given a large integer $n$, we can find integral triples $\mathbf{x}_{1},\dots, \mathbf{x}_{u}$ for which $n\equiv T(\mathbf{x}_{1})^{k}+\dots+T(\mathbf{x}_{u})^{k}\mmod{3^{k}}$ and $1\leq \mathbf{x}_{i}\leq 3^{k}.$ Fixing any one such choice of the $\mathbf{x}_{i}$, we can also find integers $x_{1},\dots,x_{s}$ satisfying
$$x_{1}^{3k}+\dots+x_{s}^{3k}=3^{-k}\Big(n-\big(T(\mathbf{x}_{1})^{k}+\dots+T(\mathbf{x}_{u})^{k}\big)\Big).$$Here the reader may find convenient to observe that the term on the right side of the equality is still large. Therefore, we obtain the representation
$$n=T(\mathbf{x}_{1})^{k}+\dots+T(\mathbf{x}_{u})^{3k}+3^{k}\big(x_{1}^{3k}+\dots+x_{s}^{3k}\big).$$ Noting that the sums of three cubes on the right side have been replaced by the specialization $3x^{3}$, one gets $G_{3}(k)\leq u+s,$ and hence by Wooley \cite[Corollary 1.2.1]{Woo7} we have that $G_{3}(k)\leq 3k\big(\log k+\log\log k+O(1)\big),$ which yields Theorem \ref{thm1.4}. As experts will realise, one could apply the ideas of Wooley \cite{Woo5} to obtain a refinement of the shape $$G_{3}(k)\leq 3k\Bigg(\log k+\log\log 3k+\log 3+2+O\Big(\frac{\log\log k}{\log k}\Big)\Bigg)$$ by using instead other specializations. For the sake of brevity though we omit making such analysis here.

Finally we combine work of this section with estimates from the previous one to analyse the singular series. Observe that by (\ref{kkk}) we can rewrite the singular series, defined in (\ref{1A}), as
$$\mathfrak{S}(n)=\sum_{q=1}^{\infty}S_{n}(q).$$ To express the above series as a product of local densities it is convenient to define the infinite sum$$\sigma(p)=\sum_{l=0}^{\infty}S_{n}(p^{l})$$ for each prime $p$. 
\begin{prop}\label{prop66}
Let $s\geq \max(5,k+2)$. Then, one has 
\begin{equation}\label{hr}\mathfrak{S}(n)=\prod_{p}\sigma(p),\end{equation}
the series $\mathfrak{S}(n)$ converges absolutely and $\frak{S}(n)\ll 1$. Moreover, if $s$ satisfy the conditions of Lemma \ref{lem9.1} one gets $\mathfrak{S}(n)\gg 1.$ 
\end{prop}
\begin{proof}
Using the estimates mentioned at the beginning of the proof of Lemma \ref{cor4} we find
\begin{equation}\label{ec4.1}\sum_{l=1}^{\infty}\lvert S_{n}(p^{l})\rvert\ll p^{1-s/2}+p^{k-s}+\sum_{l\geq k+1}l^{s}p^{l-ls/k+\varepsilon}\ll p^{-3/2}.\end{equation}
 Therefore, (\ref{hr}) holds by multiplicativity, $\mathfrak{S}(n)$ converges absolutely and $\mathfrak{S}(n)\ll 1.$ 
In order to prove the lower bound, we recall first (\ref{kkk}) and (\ref{after}) to deduce that orthogonality then yields
$$\sum_{l=0}^{h}S_{n}(p^{l})=M_{n}(p^{h})p^{h(1-3s)}.$$
 If $s$ satisfy the conditions of Lemma \ref{lem9.1}, an application of Hensel's Lemma gives the lower bound $M_{n}(p^{h})\geq p^{(3s-1)(h-\gamma)},$ which combined with the above equation implies that $\sigma(p)\geq p^{-(3s-1)\gamma}.$ Consequently, the previous estimate and (\ref{ec4.1}) deliver the lower bound $\mathfrak{S}(n)\gg 1$.
\end{proof}
\section{Approximation on the major arcs.}\label{sec6}
In this section we use a simple argument involving the Riemman-Stieltjes integral and integration by parts to approximate $f(\alpha)$ by an auxiliary function on the major arcs. We also provide bounds for this function. For such purposes, we introduce first some notation. Let $\alpha\in [0,1)$ and $a\in\mathbb{Z}$, $q\in\mathbb{N}$ with $(a,q)=1$. Denote $\beta=\alpha-a/q$ and define the aforementioned auxiliary function
\begin{equation}\label{vbet}V(\alpha,q,a)=q^{-3}S(q,a)v(\beta),\ \ \ \ \ \ \text{where}\ \ v(\beta)=\int_{[0,P]^{3}}e\big(\beta T(\mathbf{x})^{k}\big)d\mathbf{x}.\end{equation}
\begin{lem}\label{lem6}
Let $q<P$. Then one has that
\begin{equation*}\label{fVV}f(\alpha)=V(\alpha,q,a)+O\big(P^{2}q(1+n\lvert \beta\rvert)\big).\end{equation*}
\end{lem}
\begin{proof}
Before embarking on our task, it is convenient to define the sums
 \begin{equation*}K_{\mathbf{r}}(\beta)=\displaystyle\sum_{\substack{\mathbf{x}\leq P\\ \mathbf{x}\equiv\mathbf{r}\pmod{q}}}e\big(F_{\beta}(\mathbf{x})\big),\ \ \ \ \ \ \ B_{r}(x)=\displaystyle\sum_{\substack{0<z\leq x\\ z\equiv r\pmod{q}}}1,\end{equation*} where $F_{\beta}(\mathbf{x})=\beta T(\mathbf{x})^{k}$. Observe that by sorting the summation into arithmetic progressions modulo $q$ we find that
\begin{equation}\label{halpa}
f(\alpha)=\sum_{\mathbf{r}\leq q}e_{q}\big(aT(\mathbf{r})^{k}\big)K_{\mathbf{r}}(\beta).
\end{equation} 
For each $\mathbf{r}\in\mathbb{N}^{3}$ write $\mathbf{r}=(\mathbf{r}_{1},r_{3}).$ Then by Abel's summation formula we obtain \begin{align*}
K_{\mathbf{r}}(\beta)=&
B_{r_{3}}(P)\sum_{\mathbf{x}_{1}}e\big(F_{\beta}(\mathbf{x}_{1}, P)\big)-\int_{0}^{ P}\sum_{\mathbf{x}_{1}}\frac{\partial }{\partial z}e\big(F_{\beta}(\mathbf{x}_{1},z)\big)B_{r_{3}}(z)dz,
\end{align*}
where $\mathbf{x}_{1}$ runs over pairs $\mathbf{x}_{1}\in [1,P]^{2}$ with $\mathbf{x}_{1}\equiv \mathbf{r}_{1}\mmod{q}$. Consequently, using the equation $B_{r_{3}}(x)=x/q+O(1)$ and integration by parts one gets
$$K_{\mathbf{r}}(\beta)= q^{-1}\int_{0}^{P}\sum_{\mathbf{x}_{1}}e\big(F_{\beta}(\mathbf{x}_{1},z)\big)dz+O\big(q^{-2}P^{2}(1+n\lvert\beta\rvert)\big).$$ We repeat the exact same procedure for the first two variables to obtain
$$K_{\mathbf{r}}(\beta)=q^{-3}v(\beta)+O\big(q^{-2}P^{2}(1+n\lvert\beta\rvert)\big),$$ whence the combination of the above expression with (\ref{halpa}) yields the result claimed above.
\end{proof}
In order to make further progress, we provide an upper bound for $v(\beta)$ in the following lemma. This lemma will be used throughout the major arc analysis.
\begin{lem}\label{cota}
Let $\beta\in\mathbb{R}.$ One has that
$$v(\beta)\ll \frac{P^{3}}{(1+n\lvert\beta\rvert)^{1/k}}.$$
\end{lem}
\begin{proof}
Let $\mathbf{y}\in [0,P]^{2}$ and set $C_{\mathbf{y}}=y_{1}^{3}+y_{2}^{3}$. Define the auxiliary function $B_{\mathbf{y}}(y)=(3k)^{-1}y^{1/k-1}(y^{1/k}-C_{\mathbf{y}})^{-2/3}$. Note that by a change of variables one can rewrite $v(\beta)$ as 
\begin{equation}\label{bri}v(\beta)=\int_{\mathbf{y}\in[0,P]^{2}}\int_{N_{\mathbf{y}}}^{M_{\mathbf{y}}}B_{\mathbf{y}}(y)e(\beta y)dyd\mathbf{y},\end{equation} where
$N_{\mathbf{y}}=C_{\mathbf{y}}^{k}$ and $M_{\mathbf{y}}=(P^{3}+C_{\mathbf{y}})^{k}.$
Observe first that when $\lvert\beta\rvert\leq n^{-1}$ then one trivially gets
$$v(\beta)\ll \int_{\mathbf{y}\in[0,P]^{2}}\int_{N_{\mathbf{y}}}^{M_{\mathbf{y}}}B_{\mathbf{y}}(y)dyd\mathbf{y}\ll P^{3}.$$

For the case $\lvert\beta\rvert>n^{-1}$ we split the integral into
$$v(\beta)\ll I_{1}+I_{2}+I_{3},$$ where $$I_{i}=\int_{\mathbf{x}\in\mathcal{T}_{i}}e\big(\beta T(\mathbf{x})^{k}\big)d\mathbf{x}\ \ \ \ \text{for each $i\in\{1,2,3\}$}$$ and the sets of integration taken are
$$\mathcal{T}_{1}=\Big\{\mathbf{x}\in[0,P]^{3}:\ \ \mathbf{x}\leq \lvert\beta\rvert^{-1/3k}\Big\},\ \ \ \mathcal{T}_{2}=\Big\{\mathbf{x}\in[0,P]^{3}:\ \ x_{2},x_{3}> \lvert\beta\rvert^{-1/3k}\Big\},$$ $$\mathcal{T}_{3}=\Big\{\mathbf{x}\in[0,P]^{3}:\ \ x_{3}> \lvert\beta\rvert^{-1/3k},\ \ \ \ \ \ x_{1},x_{2}\leq \lvert\beta\rvert^{-1/3k}\Big\}.$$Note that for the first integral one has $I_{1}\leq \lvert \mathcal{T}_{1}\rvert\ll \lvert\beta\rvert^{-1/k}.$ For the other two it is convenient to define the parameter $T_{\beta}=(\lvert\beta\rvert^{-1/k}+C_{\mathbf{y}})^{k}$ and consider the set $\mathcal{M}_{2}=[0,P]\times[\lvert\beta\rvert^{-1/3k},P]$. Then, by applying integration by parts we find that
$$I_{2}=\int_{\mathbf{y}\in\mathcal{M}_{2}}\int_{T_{\beta}}^{M_{\mathbf{y}}}B_{\mathbf{y}}(y)e(\beta y)d\mathbf{y}dy\ll \lvert\beta\rvert^{-1}\int_{\mathbf{y}\in\mathcal{M}_{2}}B_{\mathbf{y}}(T_{\beta})d\mathbf{y},$$ where we used the fact that the function $B_{\mathbf{y}}(y)$ is decreasing. Observe that whenever $\mathbf{y}\in \mathcal{M}_{2}$ then one has $\lvert\beta\rvert^{-1/k}\leq C_{\mathbf{y}}$, which delivers the estimate
$$I_{2}\ll \lvert\beta\rvert^{-1+2/3k}\int_{\mathbf{y}\in\mathcal{M}_{2}}C_{\mathbf{y}}^{1-k}d\mathbf{y}\ll \lvert\beta\rvert^{-1+2/3k}\int_{\lvert\beta\rvert^{-1/3k}\leq x}x^{4-3k}dx\ll \lvert\beta\rvert^{-1/k}.$$ Likewise, we introduce the set $\mathcal{M}_{3}=[0,\lvert\beta\rvert^{-1/3k}]^{2}$ to handle $I_{3}$. Using the same argument we get that
$$I_{3}\ll \lvert\beta\rvert^{-1}\int_{\mathbf{y}\in\mathcal{M}_{3}}B_{\mathbf{y}}(T_{\beta})d\mathbf{y},$$
and applying the fact that $C_{\mathbf{y}}\ll \lvert\beta\rvert^{-1/k}$ whenever $\mathbf{y}\in\mathcal{M}_{3}$ to bound $B_{\mathbf{y}}(T_{\beta})$ then we obtain
$$I_{3}\ll  \lvert\beta\rvert^{-1}\int_{\mathbf{y}\in\mathcal{M}_{3}}\lvert\beta\rvert^{1-1/3k}d\mathbf{y}\ll \lvert\beta\rvert^{-1/k}.$$The combination of the bounds for $I_{1},$ $I_{2}$ and $I_{3}$ yields the result of the lemma.
\end{proof}
\section{The asymptotic formula for $R(n)$.}\label{sec7}
We compute the size of the singular integral and use the work and the bounds obtained in the previous sections to obtain an asymptotic formula for $R(n)$. Whenever $s\geq k+1$ define the aforementioned singular integral by $$J(n)=\displaystyle\int_{-\infty}^{\infty}v(\beta)^{s}e(-\beta n) d\beta.$$ Consider the set $$\mathcal{S}=\Big\{(\mathbf{y}_{1},y_{1},\dots,\mathbf{y}_{s},y_{s})\in\mathbb{R}^{3s}:\ \ \mathbf{y}_{i}\in [0,P]^{2},\ \ \ N_{\mathbf{y}_{i}}\leq y_{i}\leq M_{\mathbf{y}_{i}}\Big\}.$$ Then, recalling (\ref{bri}) and using a change of variables it follows that
\begin{align*}J(n)&
=\displaystyle\lim_{\lambda\to\infty}\int_{-\lambda}^{\lambda}\int_{\mathbf{Y}\in\mathcal{S}}\prod_{i=1}^{s}B_{\mathbf{y}_{i}}(y_{i})e\Big(\beta\Big(\sum_{i=1}^{s}y_{i}-n\Big)\Big)d\mathbf{Y}d\beta
\\
&=\lim_{\lambda\to\infty}\int_{0}^{3^{k}sn}\phi(v)\frac{\sin\big(2\pi\lambda(v-n)\big)}{\pi(v-n)}dv,
\end{align*}
where we have taken \begin{equation*}\phi(v)=\int_{\mathbf{Y}\in\mathcal{S}'}B_{\mathbf{y}_{s}}(\gamma_{v})\prod_{i=1}^{s-1}B_{\mathbf{y}_{i}}(y_{i})\ d\mathbf{Y},\ \ \ \ \ \ \ \ \ \ \text{with}\ \ \gamma_{v}=v-\displaystyle\sum_{i=1}^{s-1}y_{i}\end{equation*}  and $\mathcal{S}'\subset \mathbb{N}^{3s-1}$ is the set determined by the underlying inequalities. Observe that $\phi(v)$ is a function of bounded variation, whence by Fourier's Integral Theorem it follows that $J(n)=\phi(n)$. To obtain a precise formula for $J(n)$ it is convenient to introduce the subset $\mathcal{Y}\subset [0,n]^{s-1}$ defined by the constraint $0\leq \gamma_{n}\leq n$. Then, by several subsequent changes of variables and the formula of the Euler's Beta function one has that
\begin{align}\label{ec7.3} J(n)&
=3^{-3s}k^{-s}\Gamma\big(1/3\big)^{3s}\int_{\mathcal{Y}}\gamma_{n}^{1/k-1}\Big(\prod_{i=1}^{s-1}y_{i}^{1/k-1}\Big)\ dy_{1}\cdots dy_{s-1}\nonumber
\\
&=\Gamma\big(4/3\big)^{3s}\Gamma\big(1+1/k\big)^{s}\Gamma\big(s/k\big)^{-1}n^{s/k-1}.
\end{align}
We are now equipped to compute an asymptotic formula for the major arc contribution, which we define by
$$R_{\grM}(n)=\int_{\grM}f(\alpha)^{s}e(-\alpha n)d\alpha.$$ 
\begin{prop}\label{prop7}
Let $s\geq \max(5,k+2)$. Then,
$$R_{\grM}(n)=\Gamma\big(4/3\big)^{3s}\Gamma\big(1+1/k\big)^{s}\Gamma\big(s/k\big)^{-1}\mathfrak{S}(n)n^{s/k-1}+O(n^{s/k-1-\delta}).$$
\end{prop}
\begin{proof}
For the sake of simplicity we consider the auxiliary function $f^{*}(\alpha)$ for $\alpha\in [0,1)$ by putting $f^{*}(\alpha)=V(\alpha,q,a)$ when $\alpha\in\grM(a,q)\subset\grM$ and $f^{*}(\alpha)=0$ for $\alpha\in\grm.$ We remind the reader that $V(\alpha,q,a)$ was defined in (\ref{vbet}). Recalling Lemma \ref{lem6} then whenever $\alpha\in \grM (a,q)$ one has that
$$f(\alpha)^{s}-f^{*}(\alpha)^{s}\ll P^{2s}q^{s}(1+n\lvert \beta\rvert)^{s}+P^{2}q(1+n\lvert \beta\rvert)\lvert f^{*}(\alpha)\rvert^{s-1}.$$
Integrating over the major arcs, which were defined in (\ref{5us}), we find that
\begin{align}\label{ugu}\int_{\grM}\lvert f(\alpha)^{s}-f^{*}(\alpha)^{s}\rvert d\alpha&
\ll P^{2s+\xi(s+2)}n^{-1}+ P^{3s-3k-1+\xi}\sum_{q\leq P^{\xi}}S_{s-1}^{*}(q),
\end{align}
and hence Lemma \ref{cor4} and the assumption on $\xi$ stated before (\ref{5us}) implies that the above integral is $O(P^{3s-3k-\delta})$. Observe that by the same lemma and Lemma \ref{cota} respectively we have $$\displaystyle\sum_{ P^{\xi}< q}\lvert S_{n}(q)\rvert =O(P^{-\delta}) \ \ \ \ \ \ \ \text{and}\ \ \ \ \ \ \ \displaystyle\int_{\lvert \beta\rvert> \frac{P^{\xi}}{qn}}\lvert v(\beta)\rvert^{s}d\beta=O\big(P^{3s-3k}q^{\delta}P^{-\delta\xi}\big)$$ for $q\leq P^{\xi}$ and some $\delta>0$. Combining these observations with the aforementioned lemmata and equations (\ref{Sas}) and (\ref{ugu}) we obtain $R_{\grM}(n)=\frak{S}(n)J(n)+O(n^{s/k-1-\delta})$, and hence (\ref{ec7.3}) delivers the result.
\end{proof}
Theorem \ref{thm9.1} then follows applying Propositions \ref{prop222}, \ref{prop66} and \ref{prop7}.
\section{Asymptotic formula over the smooth numbers.}\label{Sec99}
In this section we investigate the asymptotic formula for the representation function when two of the variables of each triple lie on the smooth numbers. The strategy for bounding the integral over the minor arcs of $g(\alpha)$ combines arguments of Section \ref{sec2} with major arc techniques. 
We define the major arcs $\grN$ to be the union of 
\begin{equation}\label{kioto}\grN(a,q)=\Big\{ \alpha\in [0,1): \big\lvert \alpha-a/q\big\rvert \leq q^{-1}(\log P)^{\kappa}P^{-3k}\Big\}\end{equation} with $0\leq a\leq q\leq (\log P)^{\kappa}$ and $(a,q)=1$, where $\kappa=1/5$. We take the minor arcs $\grn=[0,1)\setminus \grN$. Similarly, when $1\leq X\leq L$ with $L=P^{1/3k},$ define $\grW(X)$ as the union of the arcs
$$\grW(\mathbf{a},q)=\Big\{\boldsymbol{\alpha}\in [0,1)^{k}: \lvert \alpha_{j}-a_{j}/q\rvert\leq q^{-1}XP^{-3j}\ \ \ \ (1\leq j\leq k)\Big\}$$ with $0\leq \mathbf{a}\leq q\leq X$ and $(q,a_{1},\dots,a_{k})=1$. For the sake of simplicity we write $$\grW=\grW(L),\ \ \ \ \ \ \grP=\grW\big((\log P)^{\kappa}\big),$$ and we take the minor arcs $\grw=[0,1)^{k}\setminus\grW$ and $\grp=[0,1)^{k}\setminus\grP.$
First we prove a lemma which permits us to have a saving over the trivial bound for some Weyl sum on $\grp$. For such purposes we consider the exponential sum associated to the Vinogradov's system
$$f_{k}(\bfalpha;X)=\sum_{1\leq x\leq X}e\big(\alpha_{1}x+\ldots +\alpha_{k} x^{k}\big).$$

\begin{lem}\label{pity}
Let $X$ be any real positive number sufficiently big in terms of $k$, let $\mu$ be a real number such that $\mu^{-1}>4k(k-1)$ and let $\gamma$ denote a real number with $X^{-\mu}\leq \gamma\leq 1.$ Then whenever $\lvert f_{k}(\bfalpha;X)\rvert\geq \gamma X$, there exist an integer $q\in\mathbb{N}$ and a tuple $\mathbf{a}=(a_{1},\dots, a_{k})\in\mathbb{N}^{k}$ with $(q,a_{1},\dots, a_{k})=1$ and $1\leq q\ll\gamma^{-k-\varepsilon}$ and such that
$$ \lvert q\alpha_{j}-a_{j}\rvert\ll\gamma^{-k-\varepsilon}X^{-j}\ \ \ \ \ (1\leq j\leq k).$$
\end{lem}

\begin{proof}
Suppose that  $\lvert f_{k}(\bfalpha;X)\rvert\geq \gamma X.$ Then, applying Wooley \cite[Theorem 1.6]{Woo6} we obtain that there exist $q\in\mathbb{N}$ and $\mathbf{a}\in\mathbb{N}^{k}$ with $(q,a_{1},\dots, a_{k})=1$ such that $1\leq q\leq X^{1/k}$ and
\begin{equation}\label{piiuuoo}\lvert q\alpha_{j}-a_{j}\rvert\leq X^{1/k-j}\ \ \ \ \ \ \ (1\leq j\leq k).\end{equation}
In order to make further progress it is convenient to define the auxiliary function 
$$T(\bfalpha;q,\mathbf{a})=q+\lvert q\alpha_{1}-a_{1}\rvert X+\dots+\lvert q\alpha_{k}-a_{k}\rvert X^{k}.$$
By Theorems 7.1, 7.2 and 7.3 of Vaughan \cite{Vau} one has that 
\begin{equation*}\lvert f_{k}(\bfalpha;X)\rvert\ll q^{\varepsilon}X T(\bfalpha;q,\mathbf{a})^{-1/k}+T(\bfalpha;q,\mathbf{a}).
\end{equation*}
Observe that equation (\ref{piiuuoo}) yields $T(\bfalpha;q,\mathbf{a})\ll X^{1/k}$, which implies that $q^{\varepsilon}X T(\bfalpha;q,\mathbf{a})^{-1/k}$ is the term dominating in the previous estimate. Therefore, by the preceding discussion and the fact that $q\leq T(\bfalpha;q,\mathbf{a})$ we obtain $$\gamma X\leq \lvert f_{k}(\bfalpha;X)\rvert\ll XT(\bfalpha;q,\mathbf{a})^{-1/k+\varepsilon},$$ which gives $T(\bfalpha;q,\mathbf{a})\ll\gamma^{-k-\varepsilon}$ and delivers the lemma.
\end{proof}
We are now equipped to prove the minor arc estimate. Recalling (\ref{Rseta}) and the definition of $\mathcal{C}(P)$ made after (\ref{rap}), observe that by orthogonality one has that
\begin{equation*}\label{Reta}R_{\eta}(n)=\int_{0}^{1}g(\alpha)^{s}e(-\alpha n)d\alpha, \ \ \ \ \ \ \ \ \ \ \ \text{where}\ \ g(\alpha)=\sum_{\mathbf{x}\in \mathcal{C}(P)}e\big(\alpha T(\mathbf{x})^{k}\big).\end{equation*}
In what follows we show that the minor arc contribution is smaller than the expected main term by combining the previous lemma and other minor arc estimates with some major arc ideas.
\begin{prop}\label{prop123} Whenever $s$ is any positive integer with $s\geq H(k)$ one has
$$\int_{\grn}\lvert g(\alpha)\rvert^{s} {\rm     d}\alpha\ll P^{3s-3k}(\log P)^{-\delta}.$$
\end{prop}

\begin{proof}
We write $H(k)=2t$ for some positive integer $t$. Recalling (\ref{GF}) and using the same argument as in (\ref{puu}) it follows that
\begin{align}\label{tiio}
\int_{\grn}\lvert g(\alpha)\rvert^{2t}d\alpha\ll P^{4t+3k(k-1)/2}\int_{\grn}\int_{[0,1)^{k-1}}\lvert F(\bfalpha)\lvert^{2t}d\bfalpha.
\end{align}
Observe that we can estimate the above integral by\begin{equation}\label{Falp}\int_{\grn}\int_{[0,1)^{k-1}}\lvert F(\bfalpha)\lvert^{2t}d\bfalpha\ll \int_{\grw}\lvert F(\bfalpha)\lvert^{2t}d\bfalpha+\int_{\grW\setminus \grP}\lvert F(\bfalpha)\lvert^{2t}d\bfalpha.\end{equation}
Note as well that combining Vaughan \cite[Theorem 5.2]{Vau} with Bourgain's result on Vinogradov's mean value theorem \cite[Theorem 1.1]{Bou} one has that
\begin{equation*}\sup_{\bfalpha\in\grw}\lvert F(\bfalpha)\rvert\ll P^{1-\delta}. \end{equation*}We then bound the first term on the right side of (\ref{Falp}) via an application of the above pointwise bound and Wooley \cite[Theorem 14.5]{Woo4} in the same way as in (\ref{BF}) to obtain $$\int_{\grw}\lvert F(\bfalpha)\lvert^{2t}d\bfalpha\ll P^{2t-3k(k+1)/2-\delta}.$$

In order to estimate the second one we provide a major arc analysis. For such purpose we consider the auxiliary function $V(\bfalpha;q,\mathbf{a})=q^{-1}S(q,\mathbf{a})I(\bfalpha-\mathbf{a}/q),$ where
$$S(q,\mathbf{a})=\sum_{r=1}^{q}e_{q}\big(a_{1}r^{3}+\ldots +a_{k}r^{3k}\big),\ \ \ \ \ \ I(\bfbeta)=\int_{0}^{P}e(\beta_{1}\gamma^{3}+\ldots+\beta_{k}\gamma^{3k})d\gamma.$$ For the sake of conciseness, we define for $\bfalpha\in [0,1)^{k}$ the function $V(\bfalpha)=V(\bfalpha;q,\mathbf{a})$ when $\bfalpha\in \grW(\mathbf{a},q)\subset \grW$ and $V(\bfalpha)=0$ for $\bfalpha\in\grw.$ Observe that by Vaughan \cite[Theorem 7.2]{Vau} then whenever $\bfalpha\in\grW$ it follows that
$$F(\bfalpha)-V(\bfalpha)\ll L,$$ whence the triangle inequality then yields $$\lvert F(\bfalpha)\rvert^{2t-2}-\lvert V(\bfalpha)\rvert^{2t-2}\ll LP^{2t-3}.$$ Therefore, combining the fact that ${\rm mes}(\grW)\ll L^{k+1}P^{-3k(k+1)/2}$ with the above estimate we get
$$\int_{\grW}\lvert F(\bfalpha)\rvert^{2t-2}d\bfalpha-\int_{\grW}\lvert V(\bfalpha)\rvert^{2t-2}d\bfalpha\ll P^{2t-2-3k(k+1)/2-\delta}.$$
On the other hand, note that Vaughan \cite[Theorems 7.1, 7.3]{Vau} gives
$$V(\bfalpha)\ll q^{\varepsilon}P\Big(q +\lvert q\alpha_{1}-a_{1}\rvert P^{3}+\cdots+\lvert q\alpha_{k}-a_{k}\rvert P^{3k}\Big)^{-1/3k},$$ and consequently, it follows that $$\int_{\grW}\lvert F(\bfalpha)\rvert^{2t-2}d\bfalpha\ll P^{2t-2-3k(k+1)/2}.$$We finally apply Lemma \ref{pity} to $F(\bfalpha)$ to obtain that when $\bfalpha\in \grp$ then one has $F(\bfalpha)<P(\log P)^{-\delta}$ for some $\delta>0$. Therefore, combining this bound with the above major arc estimate we obtain $$\int_{\grW\setminus \grP}\lvert F(\bfalpha)\rvert^{2t}d\bfalpha\ll P^{2t-3k(k+1)/2}(\log P)^{-\delta}.$$ The preceding discussion and equations (\ref{tiio}) and (\ref{Falp}) imply the proposition.
\end{proof}Next we introduce some properties of the smooth numbers concerning their density and distribution over arithmetic progressions which will be used throughout the argument for the approximation of $g(\alpha)$ over the major arcs. For such purposes, it is convenient to define
$$A_{r}(m)=\sum_{\substack{x\in \mathcal{A}(m,P^{\eta})\\ x\equiv r\pmod q}}1.$$
\begin{lem}\label{lem8.1}
Let $q,m\in\mathbb{N}$ with $q\leq P^{\eta}$ and $P^{\eta}< m\leq P.$ Then for each $0\leq r\leq q-1$ one has
\begin{equation*}A_{r}(m)=q^{-1}m\rho\Big(\frac{\log m}{\eta\log P}\Big)+O\Big(\frac{m}{\log m}\Big),\end{equation*} where the function $\rho(x)$ was defined before Theorem \ref{thm9.3}.
\end{lem}
\begin{proof}
It follows from Montgomery and Vaughan \cite[Theorem 7.2]{Mon} and the argument of the proof of Vaughan \cite[Lemma 5.4]{Vau3}.
\end{proof}
We are now equipped to provide the approximation for $g(\alpha)$. In fact, we prove here a generalized version for future use in one of our forthcoming article. For such purpose, we take constants $0\leq C_{1}<C_{2}$ and $C_{3}>0$. Let $Q>0$, let $m\in\mathbb{N}$ and define the exponential sum
\begin{equation*}g_{Q,m}(\alpha)=\displaystyle\sum_{\substack{\mathbf{x}\in\mathcal{B}}}e\big(\alpha T(m\mathbf{x})^{k}\big),\end{equation*} where $$\mathcal{B}=\Big\{\mathbf{x}\in\mathbb{N}^{3}:\ \ C_{1}Q<x_{1}\leq C_{2}Q,\ \ \  x_{2},x_{3}\in \mathcal{A}(C_{3}Q,Q^{\eta})\Big\}.$$
Despite making the choice $\kappa=1/5$ in the definition (\ref{kioto}), the following lemma contains a result which makes no use of that choice and remains valid for the range $0<\kappa<1$.
\begin{lem}\label{lem888}Let $\alpha\in \grN(a,q)$, where $a\in\mathbb{Z}$, $q\in\mathbb{N}$ with $(a,q)=1$ and $q\leq (\log)^{\kappa}$. Let $m\in\mathbb{N}$ with $(m,q)=1$. Take $Q>0$ with the property that $mQ\asymp P$ and consider $\beta=\alpha-a/q$. Then,
$$g_{Q,m}(\alpha)=V_{Q,m}(\alpha,q,a)+O\big(E(Q)\big),$$
where we take $E(Q)=Q^{3}(\log Q)^{\kappa-1}\log\log Q$ and $$V_{Q,m}(\alpha,q,a)=q^{-3}S(q,a)\rho\big(1/\eta\big)^{2}\int_{\mathbf{x}\in\mathcal{S}_{Q}}e\big(F_{m}(\mathbf{x})\big)d\mathbf{x}$$with $$F_{m}(\mathbf{x})=\beta T(m\mathbf{x})^{k},\ \ \ \ \ \ \mathcal{S}_{Q}=\Big\{\mathbf{x}\in \mathbb{R}^{3}:\ \ C_{1}Q\leq x_{1}\leq C_{2}Q,\ \ 0\leq x_{2},x_{3}\leq C_{3}Q\Big\}.$$
\end{lem}
\begin{proof}For ease of notation, we omit the subscripts for the rest of the proof. We combine the ideas of the proof of Lemma \ref{lem6} with the analysis of the distribution of smooth numbers discussed above. For such purposes, it is convenient to define first $$K_{\mathbf{r}}(\beta;m)=\sum_{\substack{\mathbf{x}\in\mathcal{B}\\ \mathbf{x}\equiv\mathbf{r}\pmod{q}}}e\big(F_{m}(\mathbf{x})\big).$$ Observe that by sorting the summation into arithmetic progressions modulo $q$ one has
\begin{equation}\label{halp}
g_{Q,m}(\alpha)=\sum_{\mathbf{r}\leq q}e_{q}\big(aT(m\mathbf{r})^{k}\big)K_{\mathbf{r}}(\beta;m).
\end{equation} 
For each $\mathbf{r}\in\mathbb{N}^{3},$ write $\mathbf{r}=(\mathbf{r}_{1},r_{3})$ with $\mathbf{r}_{1}=(r_{1},r_{2}).$ Then, we find that
$$K_{\mathbf{r}}(\beta;m)=\sum_{\mathbf{x}_{1}}\int_{0}^{C_{3}Q}e\big(F_{m}(\mathbf{x}_{1},x)\big)dA_{r_{3}}(x),$$ where the integral on the right side is the Riemann-Stieltjes integral and $\mathbf{x}_{1}$ runs over the set $$\mathcal{C}_{\mathbf{r}_{1}}=\Big\{\mathbf{x}_{1}\in\mathbb{N}^{2},\ \ C_{1}Q< x_{1}\leq C_{2}Q,\ \ x_{2}\in\mathcal{A}(C_{3}Q,Q^{\eta}),\ \ \mathbf{x}_{1}\equiv \mathbf{r}_{1}\pmod{q}\Big\}.$$ The reader may find useful to observe that the contribution of the set $[0,Q/\log Q]$ to the above integral is $O\big(q^{-3}Q^{3}(\log Q)^{-1}\big)$, whence integration by parts yields
\begin{align*}
K_{\mathbf{r}}(\beta;m)=&
A_{r_{3}}(C_{3}Q)\sum_{\mathbf{x}_{1}}e\big(F_{m}(\mathbf{x}_{1},C_{3}Q)\big)-\int_{Q/\log Q}^{C_{3}Q}\sum_{\mathbf{x}_{1}}\frac{\partial }{\partial z}e\big(F_{m}(\mathbf{x}_{1},z)\big)A_{r_{3}}(z)dz
\\
&+O\big(q^{-3}Q^{3}(\log Q)^{-1}\big).
\end{align*}
Observe that $(mQ)^{3k}\lvert\beta\rvert\ll n\lvert\beta\rvert\ll q^{-1}(\log Q)^{\kappa}$. Therefore, the integral of the error term that arises when we approximate $A_{r_{3}}(z)$ in the above equation is bounded above by $$\int_{Q/\log Q}^{C_{3}Q}\sum_{\mathbf{x}_{1}}\Big\lvert\frac{\partial F_{m}}{\partial z}(\mathbf{x}_{1},z)\Big\rvert\frac{z}{\log z}dz\ll q^{-3}Q^{3}(\log Q)^{\kappa-1}.$$ 

Before embarking in the process of giving a better description of the above equation, we recall that an application of the mean value theorem gives that for any $w\in [Q/\log Q,C_{3}Q]$ then
\begin{equation*}\label{tuii}\Big\lvert\rho\Big(\frac{1}{\eta}\Big)-\rho\Big(\frac{\log w}{\eta \log Q}\Big)\Big\rvert\ll \frac{\log\log Q}{\log Q}.\end{equation*}
Consequently, by Lemma \ref{lem8.1} and the preceding discussion we obtain 
\begin{align*}\label{Keta}K_{\mathbf{r}}(\beta;m)=&
q^{-1}\rho\big(1/\eta\big)C_{3}Q\sum_{\mathbf{x}_{1}}e\big(F_{m}(\mathbf{x}_{1},C_{3}Q)\big)
\\
&-q^{-1}\rho\big(1/\eta\big)\int_{Q/\log Q}^{C_{3}Q}z\sum_{\mathbf{x}_{1}}\frac{\partial }{\partial z}e\big(F_{m}(\mathbf{x}_{1},z)\big)dz+O\big(q^{-3}E(Q)\big),
\end{align*}
and hence integration by parts yields 
$$K_{\mathbf{r}}(\beta;m)=q^{-1}\rho\big(1/\eta\big)\int_{0}^{C_{3}Q}\sum_{\mathbf{x}_{1}}e\big(F_{m}(\mathbf{x}_{1},z)\big)dz+O\big(q^{-3}E(Q)\big),$$ where we implicitly used that the term arising after evaluating at the endpoints and the contribution of the interval $[0,Q/\log Q]$ to the above integral is $O(q^{-3}Q^{3}(\log Q)^{-1}).$ Likewise, applying a similar procedure for the second variable one gets
\begin{equation}\label{Kq}K_{\mathbf{r}}(\beta;m)=q^{-2}\rho\big(1/\eta\big)^{2}\int_{\mathbf{y}\in[0,C_{3}Q]^{2}}\sum_{x}e\big(F_{m}(x,\mathbf{y})\big)d\mathbf{y}+O\big(q^{-3}E(Q)\big),\end{equation}
where $x$ runs over the range $C_{1}Q< x\leq C_{2}Q$ and $x\equiv r_{1}\mmod{q}$.
For the first variable, we follow the same procedure as in Lemma \ref{lem6} to obtain 
$$\sum_{x\equiv r_{1}\mmod{q}}e\big(F_{m}(x,\mathbf{y})\big)=q^{-1}\int_{C_{1}Q}^{C_{2}Q}e\big(F_{m}(x,\mathbf{y})\big)dx+O\big(1+q^{-1}(\log Q)^{\kappa}\big).$$ Consequently, combining the above equation with (\ref{Kq}) we get
$$K_{\mathbf{r}}(\beta;m)=q^{-3}\rho\big(1/\eta\big)^{2}\int_{\mathbf{x}\in\mathcal{S}_{Q}}e\big(F_{m}(\mathbf{x})\big)d\mathbf{x}+O\big(q^{-3}E(Q)\big).$$Observe that since $(m,q)=1$ then a change of variables yields $S(q,am^{3k})=S(q,a).$ The lemma then follows by the preceding discussion and (\ref{halp}).
\end{proof}
\begin{cor}\label{corcor}
Let $\alpha\in \grN(a,q)$ where $a\in\mathbb{Z}$, $q\in\mathbb{N}$ with $(a,q)=1$ and $q\leq (\log P)^{\kappa}$. Consider $\beta=\alpha-a/q.$ One has that
$$g(\alpha)=W(\alpha,q,a)+O(P^{3}(\log P)^{\kappa-1}\log\log P),$$ where on recalling (\ref{vbet}) we take
$$W(\alpha,q,a)=q^{-3}S(q,a)\rho(1/\eta)^{2}v(\beta).$$
\end{cor}
\begin{proof}
Note that $g(\alpha)=g_{P,1}(\alpha)$ with the choices $C_{1}=0$, $C_{2}=1$ and $C_{3}=1$. The result is then a consequence of the previous lemma.
\end{proof}
In the rest of the section we deduce an asymptotic formula for the contribution over the major arcs. For such purposes, consider the integral
$$R_{\grN}(n)=\int_{\grN}g(\alpha)^{s}e(-\alpha n)d\alpha.$$ Define the auxiliary function $g^{*}(\alpha)$ for $\alpha\in[0,1)$ by putting $g^{*}(\alpha)=W(\alpha,q,a)$ when $\alpha\in\grN(a,q)\subset\grN$ and $g^{*}(\alpha)=0$ for $\alpha\in\grn.$ 

\begin{prop}\label{thm10.1}
Let $s\geq \max(5,k+2)$. Then, there exists a constant $\delta>0$ such that
$$R_{\grN}(n)=\Gamma\big(4/3\big)^{3s}\Gamma\big(1+1/k\big)^{s}\Gamma\big(s/k\big)^{-1}\rho\big(1/\eta\big)^{2s}\mathfrak{S}(n)n^{s/k-1}+O(n^{s/k-1}(\log n)^{-\delta}).$$
\end{prop}

\begin{proof}
Observe that by the definition (\ref{kioto}) then for $\alpha\in\grN(a,q)$ and $\beta=\alpha-a/q$ one has $$(1+n\lvert\beta\rvert)^{-(s-1)/k}\geq (\log P)^{-(s-1)\kappa/k}\geq (\log P)^{(s-1)(\kappa-1)}.$$ Consequently, Lemmata \ref{cota} and Corollary \ref{corcor} yield
$$g(\alpha)^{s}-g^{*}(\alpha)^{s}\ll P^{3s}(\log P)^{\kappa-1+\varepsilon}(1+n\lvert\beta\rvert)^{-(s-1)/k},$$
whence integrating over the major arcs we obtain that
\begin{align*}\int_{\grN}\big\lvert g(\alpha)^{s}-g^{*}(\alpha)^{s}\big\rvert d\alpha\ll P^{3s-3k}(\log P)^{-\delta}.
\end{align*} 
Observe that by Lemmata \ref{cor4} and \ref{cota} respectively we have $$\displaystyle\sum_{ q>(\log P)^{\kappa}}\lvert S_{n}(q)\rvert\ll (\log P)^{-\delta}\ \ \ \ \text{and}\ \ \ \ \ \displaystyle\int_{\lvert \beta\rvert> \frac{(\log P)^{\kappa}}{qn}}\lvert v(\beta)\rvert^{s}d\beta\ll P^{3s-3k}q^{\delta}(\log P)^{-\delta\kappa},$$ where in the second integral $q\leq (\log P)^{\kappa}$. Consequently, the combination of the aforementioned lemmas, equations (\ref{Sas}) and (\ref{ec7.3}) and the above bounds deliver the theorem.
\end{proof}
Theorem \ref{thm9.3} then follows by the application of Proposition \ref{prop66}, the estimate for the minor arcs in Proposition \ref{prop123} and the above proposition.
\section{Lower bound for $r(n)$.}\label{Sec999}
In this section we prove Theorem \ref{thm9.2} via an application of Theorem \ref{thm9.3}. The main idea is to show that the contribution to $R_{\eta}(n)$ of tuples whose coordinates have many representations as sums of three positive integral cubes is fairly small. We present first some notation and a simple lemma which will be used in the proof. Recalling the definition for $s_{3}(n)$ described before (\ref{ec12}) and the parameter $\nu$ presented right after that equation, let $\theta=\nu/ k$ and consider the set 
$$S_{K}(n)=\Big\{m\in\mathbb{N}:\ 1\leq m\leq n^{1/k}:\ s_{3}(m)>Kn^{\theta}\Big\},$$ where $K>0$.
\begin{lem}\label{lemup}
Let $n\in\mathbb{N}$ and $K>0$. Then
$$\sum_{m\in S_{K}(n)}s_{3}(m)\ll K^{-1}n^{1/k}.$$
\end{lem}
\begin{proof} It follows by noting that 
$$\sum_{m\in S_{K}(n)}s_{3}(m)\ll K^{-1}n^{-\theta}\sum_{m\in S_{K}(n)}s_{3}(m)^{2}\ll K^{-1}n^{1/k},$$
where in the last step we used (\ref{ec12}).
\end{proof}Define $R_{1}(n)$ as the contribution to $R_{\eta}(n)$ of tuples $\mathbf{X}\in\mathscr{C}^{s}$ for which $x_{i}\in S_{K}(n)$ for some $i$. Likewise, let $R_{0}(n)$ be the contribution to $R_{\eta}(n)$ of tuples $\mathbf{X}\in\mathscr{C}^{s}$ with $x_{i}\notin S_{K}(n)$ for every $1\leq i\leq s$. Observe that with this notation then one has
\begin{equation}\label{puff}R_{\eta}(n)=R_{0}(n)+R_{1}(n).\end{equation}
Note that by orthogonality, Theorem \ref{thm9.3} and Lemma \ref{lemup} one finds that
\begin{align*}R_{1}(n)&
\ll \sum_{m\in S_{K}(n)}s_{3}(m)\int_{0}^{1}g(\alpha)^{s-1}e\big(-\alpha (n-m^{k})\big)\ d\alpha\ll n^{(s-1)/k-1}\sum_{m\in S_{K}(n)}s_{3}(m)
\\
&\ll K^{-1}n^{s/k-1}.
\end{align*}
whenever $s-1\geq H(k)$. Therefore, taking $K$ to be big enough in terms of $k$ and $s$, we have by Theorem \ref{thm9.3} and (\ref{puff}) that $R_{\eta}(n)\asymp R_{0}(n)$. Since each representation of $n$ as a sum of $k$-th powers of elements of $\mathscr{C}$ is counted at most $s_{3}(x_{1})\dots s_{3}(x_{s})\leq K^{s}n^{s\theta}$ times by $R_{0}(n)$, we find that
$$r(n)\gg n^{s/k-1-s\theta}=n^{(1-\nu)s/k-1},$$ which delivers Theorem \ref{thm9.2}.
\appendix
\section{Asymptotic formula for small powers}
We improve the constraint on the number of variables in Theorem \ref{thm9.1} for the cases $2\leq k\leq 7$ by interpolating between some restriction estimates and mean value bounds for Weyl sums over the minor arcs computed in Proposition \ref{prop222}. In the following lemma we present the aforementioned restriction estimate bounds, but first define $r(k)=2^{k}$ for $2\leq k\leq 3$ and $r(k)=k(k+1)$ when $4\leq k\leq 7.$ For the sake of conciseness, we omit writing the dependence on $k$ for the rest of the section.
\begin{lem}\label{lem123}
One has that
$$\int_{0}^{1}\lvert f(\alpha)\rvert^{r}{\rm     d}\alpha\ll P^{13r/4-3k+\varepsilon}.$$
\end{lem}
\begin{proof}
Note that recalling the definition of $r_{3}(n)$ before (\ref{Rs}) we can rewrite the exponential sum $f(\alpha)$ as $$f(\alpha)=\displaystyle\sum_{x\leq 3P^{3}}r_{3}(x)e(\alpha x^{k}).$$
We then apply mean value estimates of Bourgain \cite[(1.6)]{Bou1} when $k=2$, Hughes and Wooley \cite[Theorem 4.1]{Hug} for the case $k=3$ and Wooley \cite[Corollary 1.4]{Woo4} when $4\leq k\leq 7$ to obtain
\begin{align*}\int_{0}^{1}\lvert f(\alpha)\rvert^{r}d\alpha\ll  P^{3r/2-3k+\varepsilon}\Big(\sum_{1\leq m\leq 3P^{3}} r_{3}(m)^{2}\Big)^{r/2}.\end{align*} Observe that the cited result for the case $4\leq k\leq 7$ is the weighted version of Vinogradov's mean value theorem. As experts will realise, we can apply such result to obtain the estimate that we use herein via a similar argument than the one used in (\ref{puu}). The lemma then follows by combining the previous bound with (\ref{rap}).
\end{proof}
Before describing the rest of the proof it is convenient to introduce some parameters. Take $h=\lfloor (k+1)/2\rfloor.$ Consider $p=1+r/4\xi_{0}$ and the exponents $q=p/(p-1)$ and $t=r/p+3k(3k+1)/q$, where $\xi_{0}$ is defined as the positive root of the quadratic equation obtained imposing the condition $\xi_{0}=1-1/(t-2h+1)$. Let $s=\lceil t\rceil.$ Both the values of $s$ and $t$ are gathered in Table \ref{chart2}. The following statement improves the number of variables obtained in Proposition \ref{prop222} by interpolating the estimates that we get in the second part of Proposition \ref{prop222} with Lemma \ref{lem123}. 
\begin{table}\label{table2}
\begin{tabular}{ | m{1cm} | m{1.5cm}| m{1.5cm}  | m{1.5cm}  | m{1.5cm}  | m{1.5cm} | m{1.5cm} |} 
\hline
$k$ &  $2$ & $3$ & $4$ & $5$ & $6$ & $7$ \\ 
\hline
$s$ &  $24$ & $63$ & $134$ & $216$ & $316$ & $435$ \\ 
\hline
$t$ &  $23.4331$ & $62.9722$ & $133.4783$ & $215.3978$ & $315.9897$ & $434.9924$ \\ 
\hline
\end{tabular}
\caption{}
\label{chart2}
\end{table}
\begin{prop}\label{prop24}
One has that
$$\int_{\grm}\lvert f(\alpha)\rvert^{s} {\rm     d}\alpha\ll P^{3s-3k-\delta},$$
where we take the minor arcs $\grm$ to be as described right after (\ref{5us}) with $\xi$ on the range $\xi_{0}<\xi<1-1/(s-2h+1)$.
\end{prop}
\begin{proof} By H\"older's inequality, Proposition \ref{prop222} and Lemma \ref{lem123} we obtain that
\begin{align*}\int_{\grm}\lvert f(\alpha)\rvert^{t} {\rm     d}\alpha\ll
& \Big(\int_{0}^{1}\lvert f(\alpha)\rvert ^{r}d\alpha\Big)^{1/p}\Big(\int_{\grm}\lvert f(\alpha)\rvert ^{3k(3k+1)}d\alpha\Big)^{1/q}
\\
&\ll \big(P^{13r/4-3k+\varepsilon}\big)^{1/p}\big(P^{27k^{2}+6k-\xi+\varepsilon}\big)^{1/q}\ll P^{3t-3k-\delta},
\end{align*}from where the lemma follows by observing that $t<s.$
\end{proof}
The rest of the appendix is devoted to make a refinement of the argument used in Proposition \ref{prop7} to enlarge the major arcs by taking $\xi$ on the range described above and win one variable for the cases $k=3,6$ and $7$. Let $q<P$. Denote by $N(q,P)$ to the number of solutions of the congruence 
$$T(\mathbf{x}_{1})^{k}+\ldots+T(\mathbf{x}_{h})^{k}\equiv T(\mathbf{y}_{1})^{k}+\ldots+T(\mathbf{y}_{h})^{k}\pmod{q},$$where $0\leq\mathbf{x}_{i},\mathbf{y}_{i}\leq P.$ By expressing $q$ as the product of prime powers, using the structure of the ring of integers of these prime powers and noting that the number of primes dividing $q$ is bounded by $q^{\varepsilon}$ we obtain $N(q,P)\ll q^{\varepsilon-1}P^{2h},$ and hence orthogonality yields
\begin{equation}\label{oror}\sum_{a=1}^{q}\lvert f(\beta+a/q)\rvert^{2h}\ll qN(q,P)\ll q^{\varepsilon}P^{2h}.\end{equation}
Now consider the difference function $D(\alpha)=f(\alpha)-f^{*}(\alpha).$ By the triangle inequality one has
$$\lvert f(\alpha)^{s}-f^{*}(\alpha)^{s}\rvert \ll F_{1}(\alpha)+F_{2}(\alpha),$$
where $F_{1}(\alpha)=\lvert f(\alpha)\rvert^{2h}\lvert D(\alpha)\rvert\big( \lvert f^{*}(\alpha)\rvert^{s-2h-1}+\lvert D(\alpha)\rvert^{s-2h-1}\big)$ and $F_{2}(\alpha)=\lvert D(\alpha)\rvert \lvert f^{*}(\alpha)\rvert^{s-1}.$ The integral over the major arcs for $F_{2}(\alpha)$ is bounded in the same way as in equation (\ref{ugu}), and by combining Lemmata \ref{lem6} and \ref{cota} with equation (\ref{oror}) we get
\begin{align*}\int_{\grM}F_{1}(\alpha) d\alpha&
\ll P^{3s-3k+\xi-1+\varepsilon}\sum_{q\leq P^{\xi}}S_{s-2h-1}^{*}(q)+P^{3s-3k+(s-2h+1)\xi-s+2h}\sum_{q\leq P^{\xi}}q^{\varepsilon-1}.
\end{align*}
Using the fact that $\xi<1-1/(s-2h+1)$ and Lemma \ref{cor4} we obtain that the previous integral is $O(P^{3s-3k-\delta}).$ Therefore, by the preceding discussion, the argument following (\ref{ugu}) and Propositions \ref{prop66} and \ref{prop24} then the conclusion of Theorem \ref{thm9.1} holds for the values of $s$ in Table \ref{chart2}.

\emph{Acknowledgements}: The author's work was supported in part by a European Research Council Advanced
Grant under the European Union’s Horizon 2020 research and innovation programme via grant agreement No. 695223 during his studies at the University of Bristol. It was completed while the author was visiting Purdue University under Trevor Wooley's supervision. The author would like to thank him for his guidance and helpful comments, the anonymous referees for useful remarks and both the University of Bristol and Purdue University for their support and hospitality.

\end{document}